\documentclass[reqno]{amsart}
\usepackage{amsmath, amssymb, amsthm, epsfig}
\usepackage{hyperref, latexsym}
\usepackage{url}
\usepackage[mathscr]{euscript}
\usepackage{mathabx}
\usepackage{color}
\usepackage{fullpage} 
\usepackage{setspace}

\allowdisplaybreaks

\def\today{\ifcase\month\or
  January\or February\or March\or April\or May\or June\or
  July\or August\or September\or October\or November\or December\fi
  \space\number\day, \number\year}

\DeclareMathOperator{\supp}{\mathrm{supp}}

 \newtheorem{theorem}{Theorem}
  
 \newtheorem{lemma}[theorem]{Lemma}
 \newtheorem{proposition}[theorem]{Proposition}
 \newtheorem{corollary}[theorem]{Corollary}
 \theoremstyle{definition}

 \theoremstyle{remark}

 \newcommand{\mc}{\mathcal}

 \newcommand{\R}{\mathbb{R}}

 \newcommand{\dt}{\text{\rm d}t}
  \renewcommand{\d}{\text{\rm d}}

 \newcommand{\dx}{\text{\rm d}x}
 
 \newcommand{\dy}{\text{\rm d}y}

\newcommand{\sumstar}{\sideset{}{^\star}\sum}

\newcommand{\hP}{\widetilde{\Phi}}

\makeatletter
\@namedef{subjclassname@2020}{\textup{2020} Mathematics Subject Classification}
\makeatother

\begin{document}
\title[Fourier optimization]{Fourier optimization and Montgomery's \\ pair correlation conjecture}
\author[Carneiro, Milinovich, and Ramos]{Emanuel Carneiro, Micah B. Milinovich, and Antonio Pedro Ramos}

\address{
The Abdus Salam International Centre for Theoretical Physics,
Strada Costiera, 11, I - 34151, Trieste, Italy}

\email{carneiro@ictp.it}

\address{Department of Mathematics, University of Mississippi, University, MS 38677 USA}

\email{mbmilino@olemiss.edu}

\address{SISSA - Scuola Internazionale Superiore di Studi Avanzati, Via Bonomea 265, 34136 Trieste, Italy}
\email{Antonio.Ramos@sissa.it}

\allowdisplaybreaks
\numberwithin{equation}{section}

\maketitle  

\begin{abstract}
Assuming the Riemann hypothesis, we improve the current upper and lower bounds for the average value of Montgomery's function $F(\alpha, T)$ over long intervals by means of a Fourier optimization framework. The function $F(\alpha, T)$ is often used to study the pair correlation of the non-trivial zeros of the Riemann zeta-function. Two ideas play a central role in our approach: (i) the introduction of new averaging mechanisms in our conceptual framework and (ii) the full use of the class of test functions introduced by Cohn and Elkies for the sphere packing bounds, going beyond the usual class of bandlimited functions. We conclude that such an average value, that is conjectured to be $1$, lies between $0.9303$ and $1.3208$. Our Fourier optimization framework also yields an improvement on the current bounds for the analogous problem concerning the non-trivial zeros in the family of Dirichlet $L$-functions. 
\end{abstract}


\section{Introduction}
\subsection{Montgomery's pair correlation conjecture} Let $\zeta(s)$ denote the Riemann zeta-function. We assume the truth of the Riemann hypothesis (RH) throughout this article. Montgomery's well-known pair correlation conjecture \cite{M} is a statement about the vertical distribution of the non-trivial zeros $\rho = \tfrac12 + i\gamma$ of $\zeta(s)$. It states that, for any fixed $\beta >0$, we have
\begin{align}\label{20230712_07:49}
\displaystyle N(\beta,T):=\!\!\!\sum_{\substack{ 0<\gamma,\gamma'\le T \\ 0<\gamma-\gamma' \le \frac{2\pi \beta}{\log T} }} 1 \ \sim \  N(T) \int_0^\beta \left\{ 1 - \Big( \frac{\sin \pi u}{\pi u}\Big)^2 \right\} \, \mathrm{d}u, \quad \text{as } T\to \infty. 
\end{align}
The double sum above runs over the ordinates $\gamma,\gamma'$ of two sets of non-trivial zeros of $\zeta(s)$, counted with multiplicity. The function $N(T)$ denotes the number of non-trivial zeros of $\zeta(s)$ with ordinates in the interval $(0,T]$, and it is known that $N(T)  \displaystyle \sim T\log T/(2\pi)$, as $T \to \infty$. Hence, the function $N(\beta,T)$ counts the number of pairs of zeros within $\beta$ times the average spacing between zeros.

\medskip

For a function $R \in L^1(\mathbb{R})$, define its Fourier transform by $\widehat{R}(\alpha) := \int_{-\infty}^\infty  e^{-2\pi i \alpha x} \, R(x)\,\mathrm{d}x.$ In order to understand a sum involving the differences $(\gamma - \gamma')$, as in \eqref{20230712_07:49}, Montgomery's idea was to consider more general versions of it, with a suitable weight to help with the decay and localize to nearby pairs of zeros. Indeed, setting $w(u):=4/(4+u^2)$, for any function $R \in L^1(\mathbb{R})$ such that $\widehat{R} \in L^1(\mathbb{R})$, Fourier inversion leads to the formula
\begin{equation}\label{inversion}
 \sum_{0<\gamma,\gamma'\le T} R\!\left((\gamma-\gamma') \frac{\log T}{2\pi} \right) w(\gamma-\gamma') = N(T) \int_{-\infty}^\infty\widehat{R}(\alpha) \, F(\alpha, T) \, \mathrm{d}\alpha,
\end{equation}
where Montgomery's function $F(\alpha, T)$, for $\alpha \in \mathbb{R}$ and $T\ge 15$, is defined by
\[
F(\alpha,T) := \frac{1}{N(T)} \sum_{0<\gamma,\gamma'\le T} T^{i \alpha (\gamma-\gamma')} \, w(\gamma-\gamma').
\] 
From the definition, it follows that  $F(\alpha, T)$ is real-valued and that $F(\alpha, T) = F(-\alpha, T)$. Moreover, since
\[
\sum_{0<\gamma,\gamma'\le T} T^{i \alpha (\gamma-\gamma')} \, w(\gamma-\gamma') = 2 \pi \int_{-\infty}^\infty e^{-4 \pi |u|} \bigg| \sum_{0<\gamma\le T} T^{i \alpha \gamma} e^{2\pi i \gamma u} \bigg|^2 \mathrm{d}u, 
\]
it follows that $F(\alpha, T) \ge 0$. In order to understand sums in the left-hand sides of \eqref{20230712_07:49} and \eqref{inversion}, one is led to study the asymptotic behavior of $F(\alpha, T)$, as $T \to \infty$. Assuming RH, it is known that
\begin{equation}\label{F formula}
F(\alpha,T) = \Big(T^{-2|\alpha|}\log T + |\alpha| \Big) \left( 1 + O\!\left( \sqrt{\frac{\log\log T}{\log T}} \right)\right), \quad \text{as } T\to \infty, 
\end{equation}
uniformly for $0\le |\alpha| \le 1$. This was proved by Goldston and Montgomery \cite[Lemma 8]{GM}, refining the original work of Montgomery \cite{M}. The error term here can be improved slightly, see \cite[Theorem 1]{BGSTB}. The asymptotic formula in \eqref{F formula} allows one to estimate the sum on the left-hand side of $\eqref{inversion}$ for $R \in L^1(\mathbb{R})$ with $\mathrm{supp}(\widehat{R}) \subset [-1,1]$. Montgomery conjectured that $F(\alpha, T)\sim 1$ for $|\alpha|>1$, uniformly for $\alpha$ in bounded intervals. This is sometimes called Montgomery's strong pair correlation conjecture. This assumption, via approximating the characteristic function of an interval by bandlimited functions, led Montgomery to his pair correlation conjecture in \eqref{20230712_07:49}.

\subsection{The average value of $F(\alpha, T)$} Assuming RH, from the work of Goldston \cite[Theorem 1]{G}, it is known that the following asymptotic average is equivalent to the validity of Montgomery's pair correlation conjecture in \eqref{20230712_07:49}:
\begin{equation}\label{20230712_08:18}
\displaystyle\frac{1}{\ell} \displaystyle \int_{b}^{b+\ell}F(\alpha,T) \,\d\alpha  \sim  1, \quad \text{as } T\to \infty\ , \text{ for any fixed} \  b\ge 1 \ {\rm and} \ \ell > 0.
\end{equation}
In \cite{CCChiM}, Carneiro, Chandee, Chirre, and Milinovich developed a systematic way to provide effective upper and lower bounds for the integrals appearing in \eqref{20230712_08:18}, for any $b\ge 1$ and $\ell > 0$, by connecting them to suitable Fourier optimization problems. In this paper, we are particularly interested in the long average regime, for which the following result was established in \cite[Corollary 2]{CCChiM}. Assuming RH, for $b \geq 1$ and large $\ell$ (uniformly on $b$ for the upper bound, and with $\ell \geq \ell_0(b)$ for the lower bound), one has 
\begin{equation}\label{20201215_00:21}
0.9278 + o(1) <  \frac{1}{\ell} \int_b^{b+\ell}  F(\alpha,T) \, \d\alpha  < 1.3302 + o(1)\,,
\end{equation}
as $T \to \infty$. This sharpened previous results obtained by Goldston \cite[Lemma A]{G2} and Goldston and Gonek \cite[Lemma]{GG}, that had estimates with $1/3$ in place of $0.9278$ in the lower bound and $2$ in place of $1.3302$ in the upper bound.

\smallskip

The main purpose of this paper is to provide an improvement of the asymptotic bounds in \eqref{20201215_00:21}. Although the proposed gain might seem modest at a first glance, it is conceptually interesting for it arises from different Fourier optimization problems. Two ideas play a central role in this paper: (i) the introduction of new averaging mechanisms in our conceptual Fourier optimization framework and (ii) the full use of the class of test functions introduced by Cohn and Elkies \cite{CE} for the sphere packing bounds, going beyond the usual class of bandlimited functions. This larger class of test functions has already proved useful to sharpen some bounds in the theory of the Riemann zeta-function in the work of Chirre, Gon\c{c}alves, and de Laat \cite{CGL} and of Bui, Goldston, Milinovich, and Montgomery \cite{BGMM}.

\smallskip

The following universal constant appears in our main results and plays a key role in our approach: 
\begin{equation}\label{20230709_16:18}
c_0:= \min_{x \in \R} \frac{\sin x}{x} = -0.217233... \, .
\end{equation}

\begin{theorem}\label{main}
Assume RH, let $b \geq 1$, and let $\varepsilon>0$ be arbitrary. Then, for large $\ell$, one has
\begin{equation*}
1 + c_0 \, ({\bf C_1} -1) - \varepsilon  + o(1) < \frac{1}{\ell} \int_b^{b+\ell}  F(\alpha,T) \, \d\alpha  <  {\bf C_1} + \varepsilon + o(1),
\end{equation*}
as $T \to \infty$, with $\ell \geq \ell_0(\varepsilon)$ for the upper bound and $\ell \geq \ell_0(b,\varepsilon)$ for the lower bound, where the constant ${\bf C_1}$ is defined in \eqref{20230710_14:44}. 
\end{theorem}

It follows from Lemma \ref{Lem11_CG} and Proposition \ref{Prop12} that ${\bf C_1} < 1.3208$. Using \eqref{20230709_16:18} and this upper bound for ${\bf C_1}$ in Theorem \ref{main}, one is led to the following numerical improvements of the estimates in \eqref{20201215_00:21}. 

\begin{corollary}\label{Thm1}
Assume RH and let $b \geq 1$. Then, for large $\ell$ 
one has
\begin{equation*}
0.9303 + o(1) < \frac{1}{\ell} \int_b^{b+\ell}  F(\alpha,T) \, \d\alpha  < 1.3208 + o(1),
\end{equation*}
as $T \to \infty$, uniformly on $b$ for the upper bound and with $\ell \geq \ell_0(b)$ for the lower bound.
\end{corollary}

%
%
%
%

\subsection{Working under GRH} Under the generalized Riemann hypothesis (GRH) for Dirichlet $L$-functions, more can be said about Montgomery's function $F(\alpha, T)$. This additional piece of information comes from the work of Goldston, Gonek, \"{O}zl\"{u}k, and Snyder \cite[Theorem]{GGOS}, who showed that, for any $\varepsilon >0$, one has 
\begin{equation}\label{20230708_10:21}
F(\alpha, T) \geq \frac{3}{2} - |\alpha| - \varepsilon
\end{equation}
uniformly for $1 \leq  |\alpha| \leq \frac{3}{2}  - 2\varepsilon$ and all $T \geq T_0(\varepsilon)$. Incorporating the lower bound \eqref{20230708_10:21} in our Fourier optimization framework, we obtain the following refinement of Theorem \ref{main}.
\begin{theorem}\label{main2_0901}
Assume GRH for Dirichlet $L$-functions, let $b \geq 1$, and let $\varepsilon>0$ be arbitrary. Then, for large $\ell$, one has
\begin{equation*}
1 + c_0 \, ({\bf C_1^*} -1) - \varepsilon  + o(1) < \frac{1}{\ell} \int_b^{b+\ell}  F(\alpha,T) \, \d\alpha  <  {\bf C_1^*} + \varepsilon + o(1),
\end{equation*}
as $T \to \infty$, with $\ell \geq \ell_0(\varepsilon)$ for the upper bound and $\ell \geq \ell_0(b,\varepsilon)$ for the lower bound, where the constant ${\bf C_1^*}$ is defined in \eqref{c star}. 
\end{theorem}
From Lemma \ref{Lem11_CG} and Proposition \ref{Prop12} one observes that ${\bf C_1^*} < 1.3155$. Using \eqref{20230709_16:18} and this upper bound for ${\bf C_1^*}$ in Theorem \ref{main2_0901}, one arrives at the following refinement of Corollary \ref{Thm1}.

\begin{corollary}\label{Thm2}
Assume GRH for Dirichlet $L$-functions and let $b \geq 1$. Then, for large $\ell$, 
one has
\begin{equation*}
0.9314 + o(1) < \frac{1}{\ell} \int_b^{b+\ell}  F(\alpha,T) \, \d\alpha  < 1.3155 + o(1),
\end{equation*}
as $T \to \infty$, uniformly on $b$ for the upper bound and with $\ell \geq \ell_0(b)$ for the lower bound.
\end{corollary}


\subsection{Analogues for families of Dirichlet $L$-functions} Montgomery \cite{M} also suggested the investigation of the pair correlation of zeros of the family of Dirichlet $L$-functions in $q$-aspect. The works \cite{CLLR, Ozluk} consider such a problem, studying the distribution of non-trivial zeros of $L(s,\chi)$ with two averages: the classical one over characters $\chi \ ({\rm mod} \ q)$, and another one over the modulus $q$ in a certain range. With this additional average over the modulus $q$, the authors in \cite{CLLR, Ozluk} were able to arrive at an asymptotic description as in \eqref{F formula}, now in the larger range $|\alpha| <2$, and used this information to obtain lower bounds for the proportion of simple zeros of a family of Dirichlet $L$-functions; see also the works \cite{CCLM, CGL, Sono}.

\medskip

Let us briefly describe the setup of Chandee, Lee, Liu, and Radziwi\l\l  \, \cite{CLLR} for this problem. Assume GRH for Dirichlet $L$-functions. Let $\Phi$ be a function which is real and compactly supported in $(a,b)$ with $0< a< b$. Define its Mellin transform $\hP$ by $ \hP(s) = \int_0^\infty \Phi(x)\,x^{s-1} \> \dx. $
Suppose that $\Phi(x) = \Phi(x^{-1})$ for all $x \in \R\setminus\{0\}$, $\hP(it) \geq 0$ for all $t \in \R$, and that $\hP(it) \ll |t|^{-2}$ as $|t| \to \infty$. For example, one may choose 
\[
\hP (s) = \left( \frac{e^{s} \!-\! e^{-s} }{2 s} \right)^2,
\]
so that $ \hP ( i t)  = (\sin t/ t)^2 \geq 0$, and 
\begin{align*}
\Phi(x) & =     \begin{cases}
    \frac12 - \frac14 \log x,  & \text{ for } 1 \leq x \leq e^2, \\
\frac12+ \frac14 \log x,   & \text{ for }  e^{-2} \leq x \leq 1, \\
0, & \text{ otherwise}.
\end{cases}
\end{align*} 
Let $W$ be a smooth and non-negative function, with compact support in $(1, 2)$. We define the $q$-analogue of the quantity $N(T)$ by
\[
N_{\Phi}(Q) := \sum_{q} \frac{W(q/Q)}{\varphi(q)}{\sumstar_{\chi \,(\text{mod }{q})}} 
\sum_{\gamma_{\chi}}
|\hP (i\gamma_{\chi})|^2.
\]
Here the superscript $\star$ indicates that the sum is restricted to primitive characters $\chi \,(\text{mod }{q})$, and the last sum is over all non-trivial zeros $\frac12 + i \gamma_{\chi}$ of $L(s, \chi)$, counted with multiplicity. Define the $q$-analogue of Montgomery's function $F(\alpha, T)$ by 
\begin{align}\label{20230712_14:47}
F_{\Phi}(\alpha,Q) := \frac{1}{N_\Phi (Q)}  \sum_{q} \frac{W(q/Q)}{\varphi(q)} {\sumstar_{\chi \,(\text{mod }{q})}}  \left| \sum_{\gamma_\chi } \hP \left( i\gamma_\chi \right)Q^{i \alpha \gamma_\chi }\right|^2.
\end{align}
Chandee, Lee, Liu, and Radziwi\l\l  \,\cite{CLLR} proved an asymptotic formula for $F_{\Phi}(\alpha,Q)$ similar to \eqref{F formula} for $\alpha \in (-2,2)$, showing in particular that $F_{\Phi}(\alpha,Q) \sim 1$ when $1 \leq |\alpha| <2$ (see the precise statement in Lemma \ref{Lem8_Dirichlet} below). They conjectured, in analogy with Montgomery’s original conjecture for $F(\alpha, T)$, that one should have $F_{\Phi}(\alpha,Q) \sim 1$ for all $\alpha \geq 1$. Adapting the Fourier optimization framework of \cite{CCChiM}, E.~Quesada-Herrera \cite{QE} established effective upper and lower bounds for the integrals $\frac{1}{\ell}\int_{b}^{b+\ell}F_{\Phi}(\alpha,Q)  \,\d\alpha$ for all $b\geq 1$ and $\ell >0$. In particular, in the long average regime, she established the following inequalities. Assuming GRH, for $b \geq 1$ and large $\ell$ (uniformly on $b$ for the upper bound, and $\ell \geq \ell_0(b)$ for the lower bound), one has 
\begin{equation}\label{20230901_17:45}
0.9821 + o(1) <  \frac{1}{\ell} \int_b^{b+\ell}  F_{\Phi}(\alpha,Q) \, \d\alpha  < 1.0776 + o(1)\,,
\end{equation}
as $Q \to \infty$. Here we prove the following conceptual result.

\begin{theorem}\label{main3_0901}
Assume GRH for Dirichlet $L$-functions, let $b \geq 1$, and let $\varepsilon>0$ be arbitrary. Then, for large $\ell$, one has
\begin{equation}\label{20230901_17:40}
1 + c_0 \, ({\bf C_2} -1) - \varepsilon  + o(1) < \frac{1}{\ell} \int_b^{b+\ell}  F_{\Phi}(\alpha,Q) \, \d\alpha  <  {\bf C_2} + \varepsilon + o(1),
\end{equation}
as $T \to \infty$, with $\ell \geq \ell_0(\varepsilon)$ for the upper bound and $\ell \geq \ell_0(b,\varepsilon)$ for the lower bound, where the constant ${\bf C_2}$ is defined in \eqref{c 2}. 
\end{theorem}

From Lemma \ref{Lem11_CG} and Proposition \ref{Prop12} one observes that ${\bf C_2} < 1.0650$. Using \eqref{20230709_16:18} and this upper bound for ${\bf C_2}$ in Theorem \ref{main3_0901}, one arrives at the following refinement of \eqref{20230901_17:45}.
\begin{corollary}\label{Thm3}
Assume GRH for Dirichlet $L$-functions, and let $b \geq 1$. Then, for large $\ell$, 
one has
\begin{equation*}
0.9858 + o(1) < \frac{1}{\ell} \int_b^{b+\ell}  F_{\Phi}(\alpha,Q) \, \d\alpha  < 1.0650 + o(1),
\end{equation*}
as $Q \to \infty$, uniformly on $b$ for the upper bound and with $\ell \geq \ell_0(b)$ for the lower bound.
\end{corollary}


Note how close the upper and lower bounds in Corollary \ref{Thm3} are to the conjectured value of $1$. The conceptual gain comes from the fact that, in Theorem \ref{main3_0901}, 
we apply a new Fourier optimization framework (when compared to \cite{CCChiM, QE}) tailored to this situation, with a new averaging mechanism and the full use of the Cohn--Elkies class of test functions.

\medskip

\noindent {\sc Remark.}~Analogues of Montgomery’s function $F(\alpha, T)$ have also been studied for other families of $L$-functions. Recently, the work of Chandee, Klinger-Logan, and Li \cite[Theorem 1.1]{CKL} established an analogue of \eqref{F formula} for an average over a family of $\Gamma_1(q)$ $L$-functions, in the same larger range $|\alpha| <2$, under GRH (for both this family of automorphic $L$-functions and for Dirichlet $L$-functions). We remark that the same Fourier optimization framework of our Theorem \ref{main3_0901} works in this case, and we arrive at the same conclusion for the average value (over a long range) of the analogue of $F(\alpha, T)$ for this family; see \cite[\S 2.3]{QE} for details.

\subsection{Notation} For a Lebesgue measurable set $A$ we denote by $\chi_A$ its characteristic function and by $|A|$ its Lebesgue measure. We set $x_+ :=\max\{x,0\}$. 

\section{Fourier optimization framework}\label{Sec_FOF}
In this section, we work under RH and set up the Fourier optimization framework in order to prove Theorem \ref{main}. The available number theoretic information here is given by \eqref{F formula}. We are somewhat inspired by the framework of Carneiro, Chandee, Chirre, and Milinovich \cite{CCChiM} to produce upper and lower bounds for the integral of $F(\alpha,T)$ over bounded intervals, but there are a few important conceptual changes here (also certain notation changes to better suit our outline).

\smallskip

Throughout this paper, let $\mathcal{A}_1$ be the class of continuous, even, and non-negative functions $g \in L^1(\R)$ such that $\widehat{g}(\alpha) \leq 0$ for $|\alpha| \geq 1$. One can check, via an approximation of the identity, that if $g \in \mathcal{A}_{1}$, then $\widehat{g} \in L^1(\R)$. Define the functional $\rho_1$, acting on functions $g \in \mathcal{A}_1$, by
\begin{equation}\label{20201217_11:41}
\rho_1(g) := \widehat{g}(0) +  \int_{-1}^1 \widehat{g}(\alpha)\,|\alpha| \, \mathrm{d}\alpha.
\end{equation}
This quantity is always non-negative since $|\widehat{g}(\alpha)| \leq \widehat{g}(0)$ for all $\alpha \in \R$. In fact, \eqref{20201217_11:41} is strictly positive if $g \neq {\bf 0}$. If $g \in \mathcal{A}_1$, from \eqref{inversion}, the fact that $F$ is non-negative, and \eqref{F formula}, we observe that
\begin{align}\label{20201216_11:09}
\begin{split}
  \frac{1}{N(T)} \sum_{0<\gamma,\gamma'\le T} & g\!\left((\gamma-\gamma') \frac{\log T}{2\pi} \right)  w(\gamma-\gamma')  = \ \int_{-\infty}^{\infty}\widehat{g}(\alpha) \, F(\alpha,T) \, \mathrm{d}\alpha  \\
  & \leq \int_{-1}^{1}\widehat{g}(\alpha) \, F(\alpha,T) \, \mathrm{d}\alpha = \rho_1(g) + o(1)\,,
 \end{split}
\end{align}
as $T \to \infty$. We shall see that our whole strategy is built from inequality \eqref{20201216_11:09}.

\medskip

Let us introduce our first extremal problem.
\subsubsection*{Extremal Problem 1 {\rm (EP1)}} Find the infimum
\begin{equation}\label{20230710_14:44}
{\bf C_1} := \inf_{\substack{{\bf 0} \neq g \in \mathcal{A}_1 \\ g(0) > 0}} \\ \frac{\rho_1(g)}{g(0)}.
\end{equation}

\medskip

\noindent {\sc Remark.}~Let $\mathcal{A}^{BL}_1 \subset \mathcal{A}_1$ be the subclass of bandlimited functions in $\mathcal{A}_1$, i.e.~the functions $g \in \mathcal{A}_1$ such that $\widehat{g}$ has compact support. We note that the search for the infimum in {\rm (EP1)} can be restricted to the subclass $\mathcal{A}^{BL}_1$. In fact, given $\delta >0$, let $g \in \mathcal{A}_1$, with $g(0) >0$, be such that $\frac{\rho_1(g)}{g(0)} \leq {\bf C_1}  + \delta$. Let $\varphi$ be an even Schwartz function such that $\varphi \geq 0$, $\widehat{\varphi} \geq 0$, $\int_{\R}\varphi(x)\,\dx = 1$ and $\supp(\widehat{\varphi}) \subset [-1,1]$. Let $\varphi_{\lambda}(x) := \lambda^{-1} \varphi\big( \lambda^{-1} x\big)$ be the usual approximation of the identity and set $g_{\lambda} := g * \varphi_{\lambda} \in \mathcal{A}^{BL}_1$. Then, one can verify that $\frac{\rho_1(g_{\lambda} )}{g_{\lambda} (0)} \to \frac{\rho_1(g)}{g(0)}$ as $\lambda \to 0$, and hence, for $\lambda$ small, one gets $\frac{\rho_1(g_{\lambda} )}{g_{\lambda} (0)} \leq {\bf C_1}  + 2\delta$, as desired.

\subsection{Proof of the upper bound in Theorem \ref{main}} We first discuss upper bounds for the integral of $F(\alpha, T)$ over long intervals. Here we present a sharper strategy when compared to its counterpart for upper bounds in \cite[Problem (EP4)]{CCChiM}. As we shall see in the proof below, the introduction of a continuous averaging mechanism, rather than a discrete one, is the new insight that ultimately yields the gain (see the remark after the proof of Lemma \ref{Prop_20201217_09:51} for more details). We recast our target result as the following lemma.

\begin{lemma} \label{Prop_20201217_09:51}
Assume RH, let $b \in \R$, and let $\varepsilon >0$. Then, for $\ell \geq \ell_0(\varepsilon)$, we have
\begin{align}\label{20230710_17:38}
	\limsup_{T \to \infty} \frac{1}{\ell} \int_b^{b + \ell} F(\alpha,T) \, \d\alpha \le {\bf C_1} + \varepsilon.
\end{align}	
\end{lemma}

\begin{proof} Given $\varepsilon >0$, from the remark after (EP1), let $g \in \mathcal{A}^{BL}_1$ (with $g(0) >0$), normalized so that $g(0) =1$, be such that 
\begin{equation}\label{20230710_17:29}
\rho_1(g) \leq {\bf C_1} + \frac{\varepsilon}{2}.
\end{equation}
Assume that $\supp(\widehat{g}) \subset [-M, M]$. We structure the rest of the proof in four steps.

\medskip

\noindent {\it Step 1:~The shadow construction.}  In this first step, we introduce a function that imitates the characteristic function of an interval, in a suitable sense for our purposes. For $L > M$, define the function
\begin{align*}
G_L(\alpha) := (\widehat{g}*\chi_{[-L, L]})(\alpha) =  \int_{-L}^{L}\widehat{g} (\alpha - y) \,\dy = \int_{\alpha - L}^{\alpha + L}\widehat{g} (t) \,\dt.
\end{align*}
Observe that: 
\begin{itemize}
\item For $\alpha \in [-L +M, L-M]$, we have $[-M, M] \subset [\alpha  - L, \alpha +L]$. Hence $G_L(\alpha) = \int_{\R}\widehat{g} (t) \,\dt = g(0) = 1.$

\smallskip

\item For $\alpha \in (-\infty, -L - M] \cup [L+M, \infty)$, we have $\big|\, [-M, M] \cap [\alpha - L, \alpha +L]\,\big | = 0$. Hence $G_L(\alpha) = 0$. 

\smallskip

\item For $\alpha \in I_L:= [-L - M, -L +M] \cup [L-M, L+M]$, we have $\big|G_L(\alpha)\big| \leq \|\widehat{g}\|_1$.
\end{itemize}

\medskip

\noindent {\it Step 2:~The almost majorant.} In the construction in Step 1, let $L = \frac{\ell}{2} + M$. Then observe that 
\begin{align*}
G_L(\alpha) \geq \chi_{[-\ell/2\,,\,\ell/2]}(\alpha) - \|\widehat{g}\|_1\cdot \chi_{I_L}(\alpha).
\end{align*}
With a translation by $b + \frac{\ell}{2}$, we get
\begin{align}\label{20230710_15:47}
G_L\big(\alpha - b - \tfrac{\ell}{2}\big) \geq \chi_{[b\,,\,b + \ell]}(\alpha) - \|\widehat{g}\|_1\cdot\chi_{b + \frac{\ell}{2}+I_L}(\alpha).
\end{align}

\medskip

\noindent{\it Step 3:~The key computation.} Recall that the inverse Fourier transform of $\chi_{[-L, L]}$ is given by
\begin{equation}\label{20230710_22:34}
\widecheck{\chi}_{[-L, L]}(x) = \frac{\sin(2L \pi x)}{\pi x},
\end{equation}
which is bounded in absolute value by $2L$. Then, from  \eqref{inversion}, \eqref{20201216_11:09}, and \eqref{20230710_15:47}, we have 
\begin{align}\label{20230708_10:36}
\begin{split}
\int_b^{b + \ell}  F(\alpha, T) &  \, \d\alpha  -  \|\widehat{g}\|_1  \int_{b + \frac{\ell}{2} + I_L}  F(\alpha, T) \, \d\alpha \le  \int_{\mathbb{R}} F(\alpha, T) \,G_L\big(\alpha - b - \tfrac{\ell}{2}\big)  \, \d\alpha
\\
&= \frac{1}{N(T)} \ \  {\rm Re} \sum_{0<\gamma,\gamma'\le T} T^{ i (b + \frac{\ell}{2})(\gamma - \gamma')}  \, \big(\widecheck{\chi}_{[-L, L]} \ \cdot \ g\big)\!\left((\gamma-\gamma')\frac{\log T}{2\pi}\right) \,w(\gamma-\gamma') 
\\
&\le \frac{2L}{N(T)}   \sum_{0<\gamma,\gamma'\le T} g\!\left((\gamma-\gamma')\frac{\log T}{2\pi}\right) \,w(\gamma-\gamma')  
\\
& \leq  2L \cdot \rho_1(g) + o(1).
\end{split}
\end{align}

\smallskip

\noindent {\it Step 4:~Conclusion.} From \eqref{20230708_10:36}, we conclude that
\begin{align}\label{20230710_17:34}
\limsup_{T \to \infty} \, \frac{1}{\ell}\int_b^{b + \ell}  F(\alpha, T)   \, \d\alpha  \leq \frac{2L}{\ell} \, \rho_1(g) + \frac{1}{\ell} \limsup_{T \to \infty}\, \|\widehat{g}\|_1 \int_{b + \frac{\ell}{2} + I_L}  F(\alpha, T)\,\d\alpha.
\end{align}
Now observe that the total measure of the two intervals in $b + \frac{\ell}{2} + I_L$ is $4M$, and hence it depends only $g$ but not on $b$ or $\ell$. In this case, we already know that
\begin{equation}\label{20230710_17:20}
\limsup_{T \to \infty}\  \|\widehat{g}\|_1 \int_{b + \frac{\ell}{2} + I_L}  F(\alpha, T) \, \d\alpha = O_g(1).
\end{equation}
From \eqref{20230710_17:29}, \eqref{20230710_17:34}, and \eqref{20230710_17:20}, we arrive at \eqref{20230710_17:38}. This completes the proof of the lemma. 
\end{proof}

\smallskip

\noindent {\sc Remark.}~A crucial passage in this proof that is different from the argument in \cite[p.~21]{CCChiM} is in the definition the `almost majorant' $G_L$ of $\chi_{[-\ell/2,\ell/2]}$. In \cite[p.~21]{CCChiM}, the function that plays an analogous role has the form $\widetilde{G}(\alpha) = \sum_{j=1}^N \widehat{g}(\alpha - \xi_j)$, for a certain choice of points $\xi_j$. Moving from a discrete setting to a continuous one, i.e.~by considering an integral instead of a finite sum, we are able to average out the potential fluctuations of the function $G$, leading to the sharper extremal problem (EP1) when compared to its counterpart \cite[Problem (EP4)]{CCChiM} in the previous paper. See Figure \ref{figure1} for an illustration of this philosophy.

\begin{figure} 
\includegraphics[scale=.55]{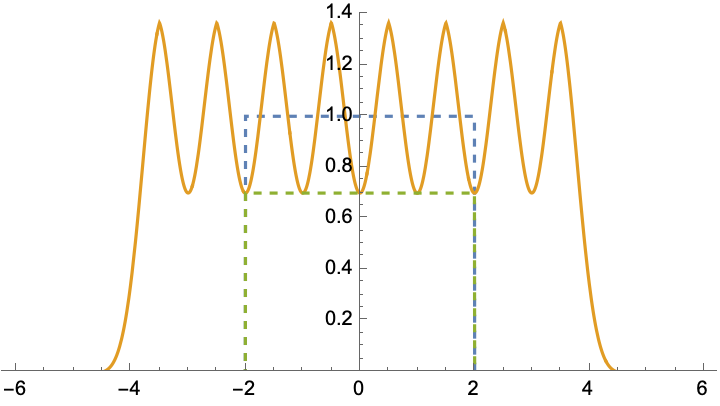} \qquad 
\includegraphics[scale=.55]{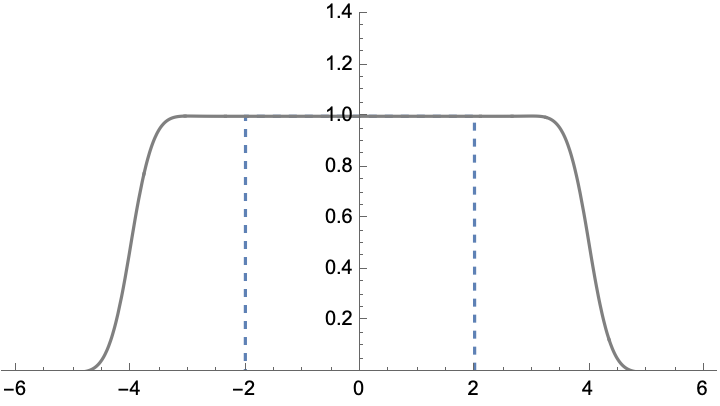} 
\caption{ Here is an illustration of why this strategy is sharper than the one in \cite{CCChiM}. Take $g \in \mathcal{A}^{BL}_1$ with $g(0) =1$ and, say, $\supp(\widehat{g}) \subset [-2, 2]$. In this particular example we consider the Fourier transform pairs $f(x) = (x^2 + 1 -\frac{1}{2\pi})\,e^{-\pi x^2}$ and $\widehat{f}(t) = (1 - t^2)\,e^{-\pi t^2}$, $K_2(x) = 2(\sin(2\pi x)/2\pi x)^2$ and $\widehat{K_2}(t) = (1 - |t/2|)_+$, and set  $g = c\, (f*K_2)$ and $\widehat{g} = c\, \widehat{f}\, \widehat{K_2}$, where $c$ is a constant so that $g(0)=1$. The graph in orange on the left represents the discrete strategy of \cite{CCChiM}, considering the sum $\widetilde{G}_4(t) = \sum_{k=0}^7\widehat{g}(t - \frac{7}{2} +k)$. The graph in gray on the right represents the continuous strategy of this paper, considering $G_4(t) :=  \int_{-4}^{4}\widehat{g} (t - y) \,\dy$. Note that $\widetilde{G}_4(t)$ roughly approximates $\chi_{[-2,2]}(t)$ but, due to its fluctuations, it is indeed only a majorant of $\lambda \, \chi_{[-2,2]}(t)$ for some $\lambda <1$, whereas $G_4(t)$ is exactly equal to $\chi_{[-2,2]}(t)$ for $-2 \leq t \leq 2$. } 
\label{figure1}
\end{figure}

\subsection{Proof of the lower bound in Theorem \ref{main}} We now turn our attention to lower bounds for the integral of $F(\alpha, T)$ over long intervals. 
We recast our target result as the next lemma, that presents a sharper strategy when compared to its counterpart for lower bounds in \cite[Problem (EP5)]{CCChiM}. As it was the case of the upper bounds, the introduction of a continuous averaging mechanism is the insight that yields the improvement over previous results. 

\begin{lemma}\label{Lem4_lower_bound}
Assume RH, let $b\ge 1$, and let $\varepsilon >0$. Then, for $\ell \geq \ell_0(b, \varepsilon)$, we have
\begin{equation}\label{20230709_17:40}
\liminf_{T \to \infty} \,\frac{1}{\ell} \int_b^{b+\ell}  F(\alpha,T) \, \d\alpha \geq 1 + c_0 \, ({\bf C_1}-1) - \varepsilon.
\end{equation}
\end{lemma}


\begin{proof} We first estimate the centered integral $\int_{-\beta}^{\beta}  F(\alpha,T) \, \d\alpha$ for $\beta=b+\ell$ and then transition to the integral $\int_b^{b+\ell}  F(\alpha,T) \, \d\alpha$. We may assume that $1 + c_0 \, ({\bf C_1}-1)- \varepsilon>0$, as otherwise the inequality in \eqref{20230709_17:40} is trivial. Given $\varepsilon >0$, from the remark after (EP1), let $g \in \mathcal{A}^{BL}_1$ (with $g(0) >0$), normalized so that $g(0) =1$, be such that 
\begin{equation}\label{20230710_22:48}
\rho_1(g) \leq {\bf C_1} + \frac{\varepsilon}{2}.
\end{equation}
Assume that $\supp(\widehat{g}) \subset [-M, M]$. We structure the rest of the proof in five steps.

\medskip

\noindent {\it Step 1:~The shadow construction.} We use the exact same construction as Step 1 in the proof of Lemma \ref{Prop_20201217_09:51}.

\medskip

\noindent {\it Step 2:~The almost minorant.} In the construction of Step 1, let $L= \beta + M$. Then observe that 
\begin{align}\label{20230709_16:17}
G_L(\alpha) \leq \chi_{[-\beta\,,\,\beta]}(\alpha) + \|\widehat{g}\|_1\cdot\chi_{I_L}(\alpha).
\end{align}

\noindent {\it Step 3:~The key computation.} We now proceed with a computation that is inspired by an argument of Goldston \cite[p.~172]{G2}, see also \cite[\textsection 2]{CCChiM}. Let $m_\gamma$ denote the multiplicity of a zero $\frac{1}{2}+i\gamma$ of $\zeta(s)$. We use \eqref{inversion}, \eqref{20201216_11:09},  \eqref{20230710_22:34} and the definition of $c_0$ in \eqref{20230709_16:18}, and  \eqref{20230709_16:17} to get (recall that $g(0) =1$)
\begin{align}\label{20230709_17:08}
\int_{-\beta}^{\beta}&   F(\alpha, T) \, \d\alpha +  \|\widehat{g}\|_1 \int_{I_L}  F(\alpha, T) \, \d\alpha \geq  \int_{\mathbb{R}}  F(\alpha, T) \,G_{L}(\alpha) \,\d\alpha \nonumber \\
& = \frac{1}{N(T)} \,\sum_{0<\gamma,\gamma'\le T} \left(\frac{2L \sin\big(L (\gamma - \gamma')\log T\big)}{L (\gamma - \gamma')\log T}\right) \,  g\!\left((\gamma-\gamma')\frac{\log T}{2\pi}\right) \,w(\gamma-\gamma')   \\
&  =   \frac{2L}{N(T)}  \left\{ \sum_{0<\gamma \le T} m_\gamma  + \sum_{\substack{0<\gamma,\gamma'\le T \\ \gamma \ne \gamma'}} \left(\frac{\sin\big(L (\gamma - \gamma')\log T\big)}{L (\gamma - \gamma')\log T}\right) \, g\!\left((\gamma-\gamma')\frac{\log T}{2\pi}\right) \,w(\gamma-\gamma')\right\} \nonumber \\
& \geq \frac{2L}{N(T)}\left\{ \sum_{0<\gamma \le T} m_\gamma  + \, c_0 \sum_{\substack{0<\gamma,\gamma'\le T \\ \gamma \ne \gamma'}}  g\!\left((\gamma-\gamma')\frac{\log T}{2\pi}\right) \,w(\gamma-\gamma')\right\} \nonumber \\
& = \frac{2L}{N(T)} \left\{ (1 - c_0) \sum_{0<\gamma \le T} m_\gamma  + \, c_0 \sum_{0<\gamma,\gamma'\le T}  g\!\left((\gamma-\gamma')\frac{\log T}{2\pi}\right) \,w(\gamma-\gamma')\right\}\nonumber \\
&\geq 2L\big(1 - c_0 + c_0\,\rho_1(g)\big) + o(1)\nonumber \\
&\geq 2\beta\big(1 - c_0 + c_0\,\rho_1(g)\big) + o(1).\nonumber 
\end{align}
Note the use of the trivial bound $\displaystyle{\sum_{0<\gamma\le T} m_\gamma \ge \sum_{0<\gamma\le T} 1 = N(T)}$ in the second to last inequality.

\medskip

\noindent {\it Step 4:~Lower bound for the centered integral.} From \eqref{20230709_17:08}, we conclude that
\begin{align}\label{20230710_22:49}
\liminf_{T \to \infty} \, \frac{1}{2\beta} \int_{-\beta}^{\beta}  F(\alpha,T) \, \d\alpha  \geq \big(1 - c_0 + c_0\,\rho_1(g)\big) - \frac{1}{2\beta} \limsup_{T \to \infty}\  \|\widehat{g}\|_1 \int_{I_L}  F(\alpha, T) \, \d\alpha.
\end{align}
Recall that the total measure of the two symmetric intervals in $I_L$ is at most $4M$, and hence it depends only $g$ but not on $\beta$. In this case we know that 
\begin{equation}\label{20230709_17:10}
\limsup_{T \to \infty}\  \|\widehat{g}\|_1 \int_{I_L}  F(\alpha, T) \, \d\alpha = O_g(1).
\end{equation}
From \eqref{20230710_22:48}, \eqref{20230710_22:49}, and \eqref{20230709_17:10}, for $\beta \geq \beta_0(\varepsilon)$, we derive that
\begin{equation}\label{centered}
\liminf_{T \to \infty} \, \frac{1}{2\beta} \int_{-\beta}^{\beta}  F(\alpha,T) \, \d\alpha \geq (1 - c_0) + c_0 \,\big({\bf C_1} + \varepsilon \big).
\end{equation}

\noindent {\it Step 5:~Conclusion.} Finally, we use the inequality \eqref{centered} to derive a lower bound for the average of $F(\alpha,T)$ over an uncentered interval. 
Fix $b \geq 1$. Letting $\beta = b + \ell$ and using the fact that $F$ is even, we see that
\begin{align*}
\frac{1}{\ell} \int_b^{b+\ell}  F(\alpha,T) \, \d\alpha = \frac{1}{2\ell} \int_{b \leq |\alpha| \leq \beta} F(\alpha,T) \, \d\alpha = \frac{2\beta}{2\ell}\,\frac{1}{2\beta} \int_{-\beta}^{\beta} F(\alpha,T) \, \d\alpha - \frac{1}{2\ell} \int_{-b}^{b} F(\alpha,T) \, \d\alpha.
\end{align*}
For sufficiently large $\ell$, the last integral on the right-hand side can be made arbitrarily small. Therefore, using the estimate in \eqref{centered}, for $\ell \geq \ell_0(b, \varepsilon)$, we arrive at
\begin{align*}
\liminf_{T \to \infty} \,\frac{1}{\ell} \int_b^{b+\ell}  F(\alpha,T) \, \d\alpha \geq 1 + c_0 \, ({\bf C_1}-1) - \varepsilon.
\end{align*}
This completes the proof of the lemma.
\end{proof}

With Lemmas \ref{Prop_20201217_09:51} and \ref{Lem4_lower_bound}, one plainly arrives at Theorem \ref{main}. A numerical upper bound for the constant ${\bf C_1}$ is discussed in Section \ref{Sec_Num}.

\section{Analogues under GRH}\label{Sec_on_GRH}

\subsection{Riemann zeta-function: proof of Theorem \ref{main2_0901}} We now set up the Fourier optimization framework in order to prove Theorem \ref{main2_0901}. Recall that, working under GRH for Dirichlet $L$-functions, we have the additional number theoretic information in \eqref{20230708_10:21}. This motivates us to define the functional $\rho_1^*$, acting on functions $g \in \mathcal{A}_1$, by
\begin{align*}
\rho_1^*(g) & := \widehat{g}(0) +  \int_{-1}^1 \widehat{g}(\alpha)\,|\alpha| \, \mathrm{d}\alpha + \int_{1 \leq |\alpha| \leq \frac32} \widehat{g}(\alpha)\left(\tfrac{3}{2} - |\alpha|\right)  \mathrm{d}\alpha.
\end{align*}
Under GRH, if $g \in \mathcal{A}_1$, we observe that 
\begin{equation}\label{20230708_10:11}
 \frac{1}{N(T)} \sum_{0<\gamma,\gamma'\le T}  g\!\left((\gamma-\gamma') \frac{\log T}{2\pi} \right)  w(\gamma-\gamma')  = \int_{-\infty}^{\infty}\widehat{g}(\alpha) \, F(\alpha,T) \, \mathrm{d}\alpha   \leq  \rho_1^*(g) + o(1) \,,
\end{equation}
as $T \to \infty$. In fact, the upper bound in \eqref{20230708_10:11} follows from the fact that, for any small $\varepsilon >0$, we have
\begin{align*}
\limsup_{T \to \infty} & \int_{-\infty}^{\infty}\widehat{g}(\alpha) \, F(\alpha,T) \, \mathrm{d}\alpha \leq \limsup_{T \to \infty} \int_{-\frac32}^{\frac32}\widehat{g}(\alpha) \, F(\alpha,T) \, \mathrm{d}\alpha \\ 
& = \widehat{g}(0) +  \int_{-1}^1 \widehat{g}(\alpha)\,|\alpha| \, \mathrm{d}\alpha + \limsup_{T \to \infty} \int_{1 \leq |\alpha| \leq \frac32}\widehat{g}(\alpha) \, F(\alpha,T) \, \mathrm{d}\alpha\\
&\leq  \widehat{g}(0) +  \int_{-1}^1 \widehat{g}(\alpha)\,|\alpha| \, \mathrm{d}\alpha + \limsup_{T \to \infty} \int_{1 \leq |\alpha| \leq \frac32 - 2\varepsilon}\widehat{g}(\alpha) \, F(\alpha,T) \, \mathrm{d}\alpha\\
&\leq  \widehat{g}(0) +  \int_{-1}^1 \widehat{g}(\alpha)\,|\alpha| \, \mathrm{d}\alpha + \int_{1 \leq |\alpha| \leq \frac32 - 2\varepsilon}\widehat{g}(\alpha)\left(\tfrac{3}{2} - |\alpha| - \varepsilon\right)\mathrm{d}\alpha\,,
\end{align*}
where we have used \eqref{F formula} and \eqref{20230708_10:21}. Sending $\varepsilon \to 0$ leads to \eqref{20230708_10:11}.

\smallskip

Using the same proofs, we can now develop the complete analogue of the Fourier optimization framework of Section \ref{Sec_FOF}, just replacing \eqref{20201216_11:09} by \eqref{20230708_10:11} where applicable. This leads us to consider the following extremal problem and to establish the following lemmas. 
\subsubsection*{Extremal Problem $1^*$ {\rm (EP1$^*$)}} Find the infimum
\begin{equation}\label{c star}
{\bf C_1^*} := \inf_{\substack{{\bf 0} \neq g \in \mathcal{A}_1 \\ g(0) > 0}} \\ \frac{\rho_1^*(g)}{g(0)}.
\end{equation}

\begin{lemma} \label{Lem5_23:09}
Assume GRH, let $b \in \R$, and let $\varepsilon >0$. Then, for $\ell \geq \ell_0(\varepsilon)$, we have
\begin{align*}
	\limsup_{T \to \infty} \frac{1}{\ell} \int_b^{b + \ell} F(\alpha,T) \, \d\alpha \le {\bf C_1^*} + \varepsilon.
\end{align*}	
\end{lemma}

\begin{lemma}\label{Lem6_lower_bound}
Assume GRH, let $b \ge 1$, and let $\varepsilon >0$. Then, for $\ell \geq \ell_0(b,\varepsilon)$, we have
\begin{equation*}
\liminf_{T \to \infty} \, \frac{1}{\ell} \int_{b}^{b+\ell}  F(\alpha,T) \, \d\alpha \geq 1  + c_0 \,\big({\bf C_1^*} -1 \big) - \varepsilon.
\end{equation*}
\end{lemma}

With Lemmas \ref{Lem5_23:09} and \ref{Lem6_lower_bound}, one plainly arrives at Theorem \ref{main2_0901}. A numerical upper bound for the constant ${\bf C_1^*}$ is discussed in Section \ref{Sec_Num}.

\subsection{Families of Dirichlet $L$-functions: proof of Theorem \ref{main3_0901}} We now set up the Fourier optimization framework in order to prove Theorem \ref{main3_0901}. We start by observing that, for any function $R \in L^1(\mathbb{R})$ such that $\widehat{R} \in L^1(\mathbb{R})$, from definition \eqref{20230712_14:47} and Fourier inversion, we have
\begin{align*}
\sum_{q} \frac{W(q/Q)}{\varphi(q)}  {\sumstar_{\chi \,(\text{mod }{q})}}  \sum_{\gamma_\chi, \gamma'_\chi } R\left( \frac{(\gamma_\chi - \gamma'_\chi)\log Q}{2\pi}\right)  \hP \left( i\gamma_\chi \right)  \hP \left( i\gamma'_\chi \right) = N_\Phi (Q) \int_{-\infty}^{\infty} F_{\Phi}(\alpha,Q) \, \widehat{R}(\alpha)\,\d\alpha.
\end{align*}
Chandee, Lee, Liu, and Radziwi\l\l \ \cite{CLLR} 
have evaluated $F_{\Phi} (\alpha, Q)$ when $|\alpha| < 2$. Their result reads as follows.
\begin{lemma}\label{Lem8_Dirichlet} \textup{(cf.  }\cite[Theorem 1.2]{CLLR}\textup{)}
Assume GRH for Dirichlet $L$-functions. Then, for any $\varepsilon > 0$, we have
\begin{align*}
  F_\Phi (\alpha, Q) =& \, \big(1+o(1)\big) \left(  f(\alpha) + \Phi\big(Q^{-|\alpha|} \big)^2 \log Q \left( \frac{1}{ 2 \pi }  \int_{-\infty}^\infty \left| \hP ( it ) \right|^2 \dt    \right)^{-1}  \right) \\
  & \quad + O\Big(  \Phi\big( Q^{- |\alpha|} \big)  \sqrt{  f(\alpha) \log Q } \Big)\,,
\end{align*} 
uniformly for $ |\alpha| \leq 2-\varepsilon$, as $ Q \to \infty$, where
$\displaystyle{  f(\alpha ) :=  \begin{cases}
| \alpha |, & \text{ for }  \,| \alpha |\leq 1, \\
1, & \text{ for }  \,  | \alpha | >1 .
\end{cases}
}$
\end{lemma}

By Plancherel’s theorem for the Mellin transform, the term
\begin{equation*}
\Phi\big(Q^{-|\alpha|} \big)^2 \log Q \left( \frac{1}{ 2 \pi }  \int_{-\infty}^\infty \left| \hP ( it ) \right|^2 \dt    \right)^{-1} 
\end{equation*}
behaves like a Dirac delta at the origin (see the argument in \cite[pp.~82--83]{CLLR}). Now let $\mathcal{A}_\Delta$ be the class of continuous, even, and non-negative functions $g \in L^1(\R)$ such that $\widehat{g}(\alpha) \leq 0$ for $|\alpha| \geq \Delta$. For $1 \leq \Delta \leq 2$, Lemma \ref{Lem8_Dirichlet} above motivates the definition of the functional $\rho_{\Delta}$, acting on functions $g \in \mathcal{A}_{\Delta}$, by
\begin{equation*}
\rho_{\Delta}(g) := \widehat{g}(0) +  \int_{-1}^1 \widehat{g}(\alpha)\,|\alpha| \, \mathrm{d}\alpha + \int_{1 \leq |\alpha|\leq \Delta} \widehat{g}(\alpha)\,\d\alpha.
\end{equation*}

Under GRH, for $1 \leq \Delta <2$ and if $g \in \mathcal{A}_\Delta$, we observe that Lemma \ref{Lem8_Dirichlet} plainly yields
\begin{align}\label{20230712_15:57}
\begin{split}
\frac{1}{N_\Phi (Q)}\sum_{q} \frac{W(q/Q)}{\varphi(q)}  {\sumstar_{\chi \,(\text{mod }{q})}}  & \sum_{\gamma_\chi, \gamma'_\chi } g\left( \frac{(\gamma_\chi - \gamma'_\chi)\log Q}{2\pi}\right)   \hP \left( i\gamma_\chi \right)  \hP \left( i\gamma'_\chi \right) =  \int_{-\infty}^{\infty} F_{\Phi}(\alpha,Q) \, \widehat{g}(\alpha)\,\d\alpha\\
& \leq \int_{-\Delta}^{\Delta} F_{\Phi}(\alpha,Q) \, \widehat{g}(\alpha)\,\d\alpha  = \rho_\Delta(g) + o(1)\,,
\end{split}
\end{align}
as $Q \to \infty$. This leads us to consider the following limiting extremal problem (for $\Delta =2$).
\subsubsection*{Extremal Problem 2 {\rm (EP2)}} Find the infimum
\begin{equation} \label{c 2}
{\bf C_2} := \inf_{\substack{{\bf 0} \neq g \in \mathcal{A}_2 \\ g(0) > 0}} \\ \frac{\rho_2(g)}{g(0)}.
\end{equation}
We now claim that we can develop the complete analogue of the Fourier optimization framework of Section \ref{Sec_FOF} and arrive at the following lemmas.

\begin{lemma} \label{Lem9_23:09}
Assume GRH, let $b \in \R$ and $\varepsilon >0$. Then, for $\ell \geq \ell_0(\varepsilon)$, we have
\begin{align*}
	\limsup_{Q \to \infty} \frac{1}{\ell} \int_b^{b + \ell} F_{\Phi}(\alpha,Q) \, \d\alpha \le {\bf C_2} + \varepsilon.
\end{align*}	
\end{lemma}

\begin{lemma}\label{Lem10_lower_bound}
Assume GRH, let $b\ge 1$, and let $\varepsilon >0$. Then, for $\ell \geq \ell_0(b,\varepsilon)$, we have
\begin{equation*}
\liminf_{Q \to \infty} \, \frac{1}{\ell} \int_{b}^{b+\ell}  F_{\Phi}(\alpha,Q) \, \d\alpha \geq 1 + c_0 \,\big({\bf C_2}-1)- \varepsilon.
\end{equation*}
\end{lemma}

We only need a minor adjustment to the proofs. Letting $\mathcal{A}_\Delta^{BL} \subset \mathcal{A}_\Delta$ be the subclass of bandlimited functions, we have already noted (see remark after problem (EP1)) that, given $\varepsilon >0$, it is possible to choose $ g \in \mathcal{A}_2^{BL}$ with $g(0)=1$ such that 
$$\rho_2(g) \leq {\bf C_2} + \frac{\varepsilon}{4}.$$
For $\Delta <2$, let $g_\Delta(x) := g(\Delta x/2)$. Then $\widehat{g_{\Delta}}(\alpha) = (2/\Delta)\,\widehat{g}(2\alpha / \Delta)$, and we see that $g_{\Delta} \in \mathcal{A}_{\Delta}^{BL}$. For $\Delta$ sufficiently close to $2$, we have 
\begin{equation}\label{20230712_16:57}
\rho_\Delta(g_{\Delta}) \leq \rho_2(g)+ \frac{\varepsilon}{4} \leq  {\bf C_2} + \frac{\varepsilon}{2}.
\end{equation}
We then start as in the proofs of Lemmas \ref{Prop_20201217_09:51} and \ref{Lem4_lower_bound}, using this test function $g_{\Delta}$ and \eqref{20230712_16:57} in place of \eqref{20230710_17:29} and \eqref{20230710_22:48}. The rest of the outline is the same, replacing \eqref{20201216_11:09} by \eqref{20230712_15:57} at the end of the key computation in Step 3.

\medskip

With Lemmas \ref{Lem9_23:09} and \ref{Lem10_lower_bound}, one plainly arrives at Theorem \ref{main3_0901}. A numerical upper bound for the constant ${\bf C_2}$ is discussed in Section \ref{Sec_Num}.

\section{Numerical approximations} \label{Sec_Num}

In the previous sections, we have seen examples of how one can set up Fourier optimization problems that are related to number theoretical situations. Other examples in the literature include \cite{CCChiM, CMRQE, CMS, CQE, CE, Ga, QE}. Quite often, it is a hard task to find the exact solutions and extremal functions for such Fourier optimization problems, and one is naturally led to the search of near-extremizers via computational methods. 

\medskip

For instance, if we restricted ourselves in (EP1) to the subclass of functions $g \in \mathcal{A}_1$ with $\supp(\widehat{g}) \subset [-1,1]$, the answer of the problem would be the classical constant $${\bf C_{MT}} := \frac{1}{2} + 2^{-\frac12} \cot\big(2^{-\frac12}\big) = 1.32749\ldots$$ 
found by Montgomery and Taylor \cite{M2} (for alternative proofs, see \cite[Corollary 14]{CCLM} or \cite[Appendix A]{ILS}). Thanks to our Fourier optimization setup, this would already be an improvement over the current best bound appearing on the right-hand side of \eqref{20201215_00:21}. Having this classical constant ${\bf C_{MT}} $ as our initial benchmark, one of the points we want to make is that the extended class $\mathcal{A}_1$ allows one to go further.

\medskip

It turns out that our extremal problems (EP1), (EP$1^*$), and (EP2) are connected to the thorough study of Chirre, Gon\c{c}alves, and de Laat \cite{CGL} on simple zeros of $\zeta(s)$ and families of $L$-functions via a pair correlation approach. In this work, similar extremal problems received a robust numerical treatment via modern semidefinite programming tools, and we shall take advantage of certain upper bounds established therein; see Lemma \ref{Lem11_CG} below. Since the notation used in \cite{CGL} is slightly different than ours, let us briefly review their setup. Chirre, Gon\c{c}alves, and de Laat in \cite{CGL} work within the linear programming class $\mathcal{A}_{LP}$, introduced by Cohn and Elkies \cite{CE} for the sphere packing bounds. Here, $\mathcal{A}_{LP}$ is the class of even and continuous functions $f \in L^1(\R)$, satisfying the following conditions:
\begin{itemize}
\item[(i)] $\widehat{f}(0) = f(0) =1$;
\item[(ii)] $\widehat{f} \geq 0$;
\item[(iii)] $f$ is eventually non-positive.
\end{itemize}
By eventually non-positive one means that $f(x) \leq 0$ for all sufficiently large $|x|$. One defines the last sign change of $f$ by 
$$r(f) = \inf\big\{ r>0 \ : \ f(x) \leq 0 \ {\rm for} \ |x|\geq r\big\}.$$
In \cite{CGL}, the authors are interested in minimizing the following three functionals over the class $\mathcal{A}_{LP}$:
\begin{align*}
\mc{Z}(f) &:= r(f) + \frac{2}{r(f)} \int_0^{r(f)} f(x) \, x\,\dx\,;\\
\widetilde{\mc{Z}}(f) &:=  r(f) + \frac{2}{r(f)} \int_0^{r(f)} f(x) \, x\,\dx + 3\int_{r(f)}^{\frac{3}{2}r(f)} f(x)\,\dx - \frac{2}{r(f)} \int_{r(f)}^{\frac{3}{2}r(f)} f(x)\,x\,\dx\,;\\
\mc{L}(f) &:= \frac{r(f)}{2} + \frac{4}{r(f)}\int_0^{\frac{r(f)}{2}} f(x)\,x\,\dx + 2\int_{\frac{r(f)}{2}}^{r(f)} f(x)\,\dx.
\end{align*}
They prove the following result.

\begin{lemma}[cf. \cite{CGL}]\label{Lem11_CG} The following upper bounds hold:
\begin{align*}
\inf_{f \in \mc{A}_{LP}} \mc{Z}(f) &< 1.3208\,;\\
\inf_{f \in \mc{A}_{LP}} \widetilde{\mc{Z}}(f) &< 1.3155\,;\\
\inf_{f \in \mc{A}_{LP}} \mc{L}(f) &< 1.0650.
\end{align*}
\end{lemma}
Our problems problems (EP1), (EP$1^*$), and (EP2) are related to these as follows. 
\begin{proposition} \label{Prop12}We have
\begin{align}
{\bf C_1} &\leq \inf_{f \in \mc{A}_{LP}} \mc{Z}(f)\,; \label{20230716_17:00}\\
{\bf C_1^*} & \leq \inf_{f \in \mc{A}_{LP}} \widetilde{\mc{Z}}(f) \,; \label{20230716_17:01}\\
{\bf C_2} &\leq \inf_{f \in \mc{A}_{LP}} \mc{L}(f). \label{20230716_17:02}
\end{align}
\end{proposition}
\begin{proof}
In order to show \eqref{20230716_17:00}, note that if $f \in \mc{A}_{LP}$ then $g(x):= \widehat{f}\big(x / r(f)\big)\in \mc{A}_1$ since $\widehat{g}(\alpha) = r(f) \, f\big( r(f) \,\alpha\big)$. Note that $g(0) = 1$ and $\widehat{g}(0)=r(f)$. Then, a change of variables yields
\begin{equation*}
\mc{Z}(f)  = \widehat{g}(0) + 2 \int_0^1 \widehat{g}(\alpha)\,\alpha\, \d\alpha = \rho_1(g).
\end{equation*}
This plainly leads to \eqref{20230716_17:00}. The same construction leads to  \eqref{20230716_17:01}. For \eqref{20230716_17:02}, if $f \in \mc{A}_{LP}$, we consider instead $g(x):= \widehat{f}\big(2x / r(f)\big)\in \mc{A}_2$.
\end{proof}

The upper bounds from Lemma \ref{Lem11_CG} and Proposition \ref{Prop12}, together with our Fourier optimization framework developed in Sections \ref{Sec_FOF} and \ref{Sec_on_GRH}, lead us to the bounds proposed in Corollaries \ref{Thm1}, \ref{Thm2}, and \ref{Thm3}.

\section*{Acknowledgments}
Part of this research took place while the second author was a Visiting Researcher at the Abdus Salam International Centre for Theoretical Physics (ICTP). He thanks ICTP for their generosity and hospitality. MBM was supported by the NSF grant DMS-2101912 and the Simons Foundation (award 712898).

\end{document}